\documentclass{amsart}       
\usepackage{graphicx}
\usepackage{amssymb}
\usepackage{enumerate}
\usepackage{textcomp}
\usepackage{graphicx,epsfig}
\usepackage{amsfonts,amssymb,enumerate}
\usepackage{latexsym,amsfonts,amssymb,amsmath,amscd,eufrak,amsthm}
\usepackage{fancybox,shadow,color,graphicx,psfrag,float}
\usepackage{pstricks,pst-node,pst-text,pst-3d,pst-plot}
\usepackage{tikz,pgffor}

\newcommand{\na}{\mathbb{N}}

\newcommand{\ff}{\mathbb{F}}

\newcommand{\m}{\mathcal}
\renewcommand{\P}{\mathbb{P}}
\renewcommand{\deg}[1]{deg\,#1}

\newcommand{\diff}{\operatorname{Diff}}

\newcommand{\supp}{\operatorname{supp}}

\newcommand{\ncon}{\operatorname{\equiv\hspace{-4.5mm}{\diagup}}}
\newtheorem{thm}{Theorem}
\newtheorem{pro}[thm]{Proposition}
\newtheorem{rem}[thm]{Remark}

\begin{document}

\title{On cubic Kummer towers of Garcia, Stichtenoth and Thomas type}

\begin{abstract}
In this paper we initiate the study of the class of cubic Kummer type towers considered by Garcia, Stichtenoth and Thomas in 1997 by classifying the asymptotically good ones in this class.
\end{abstract}

\maketitle

\section{Introduction}
\label{intro}
It is well known the importance of asymptotically good recursive towers in coding theory and some other branches of information theory (see, for instance, \cite{NX01}). Among the class of recursive towers there is an important one, namely the class  of Kummer type towers which are recursively defined by equations of the form $y^m=f(x)$ for some suitable exponent $m$ and rational function $f(x)\in K(x)$. A particular case was studied by Garcia, Stichtenoth and Thomas in \cite{GST97} where a Kummer tower over a finite field $\ff_q$ with $q\equiv 1\mod m$ is recursively defined by an equation of the form
\begin{equation}\label{kummergst97}
y^m=x^df(x)\,,
\end{equation}
where $f(x)$ is a polynomial of degree $m-d$ such that $f(0)\neq 0$ and $\gcd(d,m)=1$. The authors showed that they have positive splitting rate and, assuming the existence of a subset $S_0$ of $\ff_q$ with certain properties, the good asymptotic behavior of such  towers can be deduced together with a concrete non trivial lower bound for their limit. Later Lenstra showed in \cite{Le02} that in the case of an equation of the form \eqref{kummergst97} over a prime field, there is not such a set $S_0$ satisfying the above conditions of Garcia, Stichtenoth and Thomas. Because of Lenstra's result it seems reasonable to expect that many Kummer towers defined by equations of the form \eqref{kummergst97} have infinite genus. However, to the best of our knowledge there are not examples of such towers in the literature.
The aim of this paper is to  classify those asymptotically good Kummer type towers considered by Garcia, Stichtenoth and Thomas in \cite{GST97} recursively defined by an equation of the form
\begin{equation}\label{cubicgst}
y^3=xf(x)\,,
\end{equation}
over a finite field $\ff_q$ where $q\equiv 1\mod 3$ and $f(t)\in\ff_q[t]$ is a monic and quadratic polynomial. It was shown in \cite{GST97} that there are choices of the polynomial $f$ giving good asymptotic behavior and even optimal behavior. For instance if $f(x)=x^2+x+1$  then the equation \eqref{cubicgst} defines an optimal tower over $\ff_4$, a finite field with four elements (see \cite[Example 2.3]{GST97}). It is worth to point out that the quadratic case (i.e. an equation of the form $y^2=x(x+a)$ with $0\neq a\in \ff_q$) is already included in the extensive computational search of good quadratic tame towers performed in \cite{MaWu05}.

The organization of the paper is as follows. In Section \ref{notanddef} we give the basic definitions and we establish the notation to be used throughout the paper. In Section \ref{genus} we give an overview of the main ideas, in the general setting of towers of function fields over a perfect field $K$, used to prove the infiniteness of the genus of a tower. In Section \ref{pyramid}  we prove some criteria involving the basic function field associated to a tower to check the infiniteness of its genus. Finally in Section \ref{examples} we prove our main result (Theorem \ref{teoexe2}) where we show that asymptotically good towers defined by an equation of the form \eqref{kummergst97}
\[y^3=x(x^2+bx+c)\,,\]
with $b,c \in \ff_q$ and $q\equiv 1 \mod 3$ fall into three mutually disjoint classes according to the way the quadratic polynomial $x^2+bx+c$ splits into linear factors over $\ff_q$.  From this result many examples of non skew recursive Kummer towers with positive splitting rate and infinite genus can be given. We would like to point out that there are very few known examples showing this phenomena.  An example of a non skew Kummer tower (but not of the form  \eqref{kummergst97}) with infinite genus over a prime field $\ff_p$ was given in \cite{MaWu05} but, as we will show at the end of  Section \ref{genus}, there is a mistake in the argument used by the authors.   There are also examples of non skew Kummer towers with bad asymptotic behavior over some non-prime finite fields  given by Hasegawa in \cite{Ha05} but those Kummer towers have zero splitting rate.

\section{Notation and Definitions}\label{notanddef}
In this work we shall be concerned with \emph{ towers} of function fields and this means a sequence $\m{F}=\{F_i\}_{i=0}^{\infty}$ of function fields over a field $K$ where for each index $i\geq 0$ the field $F_i$ is a proper subfield of $F_{i+1}$, the field extension $F_{i+1}/F_i$ is finite and separable and $K$ is the full field of constants of each field $F_i$ (i.e. $K$ is algebraically closed in each $F_i$). If the genus $g(F_i)\rightarrow \infty$ as $i\rightarrow \infty$ we shall say that $\m{F}$ is a {\em tower in the sense of Garcia and Stichtenoth}.

Following \cite{Stichbook09}  (see also \cite{GS07}), one way of constructing towers of function fields over $K$ is by giving a bivariate polynomial
\[H\in K[X,Y]\,,\]
 and a transcendental element $x_0$ over $K$. In this situation a tower $\mathcal{F}=\{F_i\}_{i=0}^{\infty}$ of function fields over $K$ is defined as
 \begin{enumerate}[(i)]
 \item $F_0=K(x_0)$, and
 \item $F_{i+1}=F_i(x_{i+1})$ where $H(x_i,x_{i+1})=0$ for $i\geq 0$.
 \end{enumerate}
 A suitable choice of the bivariate polynomial $H$ must be made in order to have towers. When the choice of $H$ satisfies all the required conditions we shall say that the tower $\mathcal{F}$ constructed in this way is a {\em recursive tower} of function fields over $K$. Note that for a recursive tower $\mathcal{F}=\{F_i\}_{i=0}^{\infty}$ of function fields over $K$ we have that
  $$F_i=K(x_0,\ldots,x_i)\qquad \text{for }i\geq 0,$$
where $\{x_i\}_{i=0}^{\infty}$ is a sequence of transcendental elements over $K$.

Associated to a recursive tower  $\mathcal{F}=\{F_i\}_{i=0}^{\infty}$ of function fields $F_i$  over $K$ we have the so called {\em basic function field} $K(x,y)$ where $x$ is transcendental over $K$ and $H(x,y)=0$.

For the sake of simplicity we shall say from now on that $H$ defines the tower $\m{F}$ or, equivalently, that  tower $\m{F}$ is recursively defined by the equation $H(x,y)=0$.

A tower $\m{F}=\{F_i\}_{i=0}^{\infty}$ of function fields over a perfect field $K$ of positive characteristic is called \emph{tame} if the ramification index $e(Q|P)$ of any place $Q$ of $F_{i+1}$ lying above a place $P$ of $F_i$ is relatively prime to the characteristic of $K$ for all $i\geq 0$. Otherwise the tower $\m{F}$ is called \emph{wild}.

 The set of places of a function field $F$ over $K$ will be denoted by $\P(F)$.

The following definitions are important when dealing with the asymptotic behavior of a tower. Let $\mathcal{F}=\{F_i\}_{i=0}^{\infty}$ be a  tower of function fields
over a finite field $\ff_q$ with $q$ elements.
The {\em splitting rate} $\nu(\m{F})$ and the {\em genus} $\gamma(\m{F})$ of $\m{F}$ over $F_0$ are defined, respectively, as
$$\nu(\m{F})\colon=\lim_{i\rightarrow \infty}\frac{N(F_i)}{[F_i:F_0]}\,, \qquad\gamma(\m{F})\colon=\lim_{i\rightarrow \infty}\frac{g(F_i)}{[F_i:F_0]}\,.$$
If $g(F_i)\geq 2$
for $i\geq i_0\geq 0,$ the {\em limit} $\lambda(\m{F})$ of $\m{F}$ is defined as $$\lambda(\m{F})\colon=\lim_{i\rightarrow \infty}\frac{N(F_i)}{g(F_i)}\,.$$
It can be seen that all the above limits exist and that $\lambda(\m{F})\geq 0$ (see \cite[Chapter 7]{Stichbook09}).

Note that the definition of the genus of $\m{F}$ makes sense also in the case of a tower $\m{F}$ of function fields over a perfect field $K$.

We shall say that a tower $\mathcal{F}=\{F_i\}_{i=0}^{\infty}$ of function fields over $\ff_q$ is {\em asymptotically good} if $\nu(\m{F})>0$ and $\gamma(\m{F})<\infty$. If either $\nu(\m{F})=0$ or $\gamma(\m{F})=\infty$ we shall say that $\m{F}$ is {\em asymptotically bad}.

From the well-known Hurwitz genus formula
(see \cite[Theorem 3.4.13]{Stichbook09}) we see that the condition $g(F_i)\geq 2$ for $i\geq i_0$ in the definition of $\lambda(\m{F})$ implies that $g(F_i)\rightarrow \infty$ as $i\rightarrow \infty$.
Hence, when  we speak of the limit of a tower of function fields we are actually speaking of the limit of a tower in the sense of Garcia and Stichtenoth (see \cite[Section 7.2]{Stichbook09}).

It is easy to check that in the case of a tower $\m{F}$ we have that $\m{F}$ is asymptotically good if and only if $\lambda(\m{F})>0$. Therefore a tower $\m{F}$ is  asymptotically bad if and only if $\lambda(\m{F})=0$.

\section{The genus of a tower}\label{genus}

As we mentioned in the introduction, a simple and useful condition implying that $H \in \ff_q[x,y]$ does not give rise to an asymptotically good recursive tower  $\m{F}$ of function fields over $\ff_q$ is that $\deg_xH\neq\deg_yH$. With this situation in mind
we shall say that a recursive tower $\m{F}=\{F_i\}_{i=0}^{\infty}$ of function fields over a perfect field $K$ defined by a polynomial $H \in K[x,y]$ is {\em non skew} if $\deg_xH=\deg_yH$. In the skew case (i.e. $\deg_xH\neq \deg_yH$) we might have that $[F_{i+1}:F_i]\geq 2$ for all $i\geq 0$ and
even that $g(F_i)\rightarrow \infty$ as $i\rightarrow\infty$ but, nevertheless, $\m{F}$ will be asymptotically bad. What happens is that if $\deg_yH>\deg_xH$ then
the splitting rate  $\nu(\m{F})$ is zero (this situation makes sense in the case $K=\ff_q)$ and if  $\deg_xH>\deg_yH$ the genus $\gamma(\m{F})$ is infinite (see \cite{GS07} for details). Therefore the study of good asymptotic behavior in the case of recursive towers must be focused on non skew towers. Since the splitting rate of recursive towers defined by an equation of the form \eqref{kummergst97} is positive, their good asymptotic behavior is determined by their genus.

From now on $K$ will denote a perfect field and we recall that $K$ is assumed to be the full  field of constants of each function field $F_i$ of any given tower $\m{F}$ over $K$.  We recall a well-known formula for the genus of a tower  $\m{F}=\{F_i\}_{i=0}^{\infty}$ in terms of a subtower $\m{F}'=\{F_{s_i}\}_{i=1}^\infty$, namely
\begin{equation}\label{ecu1paper3}
\gamma(\m{F})=
\lim_{i \rightarrow \infty}\frac {g(F_{s_i})}{[F_{s_i}:F_0]}=g(F_0)-1+\frac 1 2 \sum_{i=1}^\infty
\frac{\deg{\diff(F_{s_{i+1}}/F_{s_i})}}{[F_{s_{i+1}}:F_0]}.\end{equation}
\begin{rem}\label{remarkdivisor}
Suppose now that there exist positive
functions  $c_1(t)$ and $c_2(t)$, defined for $t\geq 0$, and  a divisor  $B_i\in \m{D}(F_i)$ such that for each $i\geq 1$
 \begin{enumerate}[\text{Condition} (a):]
\item $\deg B_i\geq c_1(i)[F_i:F_0]$ and\label{thm3.2-a}
 \item $\sum\limits_{P\in supp(B_i)}\sum\limits_{Q|P}d(Q|P)\deg Q\geq c_2(i)[F_{{i+1}}\colon F_i]\deg{B_i}\,,$ \label{thm3.2-b}
 \end{enumerate}
where the inner sum runs over all places $Q$ of $F_{i+1}$ lying above $P$, then it is easy to see from \eqref{ecu1paper3} that if the series
\begin{equation}\label{thm3.2-c}
  \sum_{i=1}^{\infty}c_1(i)c_2(i)
\end{equation} is divergent then $\gamma(\m{F})=\infty$.
\end{rem}

  With the same hypotheses as in Remark~\ref{remarkdivisor}, if in addition $\m{F}=\{F_i\}_{i=0}^{\infty}$ is non skew and  recursively defined by the equation $H(x,y)=0$ such that $H(x,y)$, as a polynomial with coefficients in $K(y)$, is irreducible in $K(y)[x]$ then condition \eqref{thm3.2-a}
  can be replaced by the following
\begin{enumerate}[(a')]
\item $\deg{B_{j}}\geq c_1(j)\cdot \deg(b(x_{j}))^{j}$ where $b\in K(T)$ is a rational function and $(b(x_{j}))^{j}$ denotes either the pole divisor or the zero divisor of $b(x_{j})$ in $F_{j}$,\label{thm3.2-a'}
\end{enumerate}
and the same result hold, i.e., $\gamma(\m{F})=\infty$. These are the usual ways of proving the infiniteness of the genus of a recursive tower $\m{F}$.

In particular the existence of a divisor as in Remark \ref{remarkdivisor}  can be proved by showing that sufficiently many places of $F_i$ are ramified in  $F_{i+1}$ in the  sense that the number $r_i=\#(R_i)$ where
\[R_i=\{P\in\P(F_i)\,:\,\text{$P$ is ramified in $F_{i+1}$}\}\,.\] satisfies the estimate
\[r_i\geq c_i[F_{s_{i+1}}:F_0]\,,\]
where $c_i>0$ for $i\geq 1$ and the series
$\sum_{i=1}^{\infty}c_i$
is divergent. It is easily seen that the divisor of $F_i$
\[B_i=\sum_{P\in R_i}P\,,\]
satisfies the conditions \eqref{thm3.2-a} and \eqref{thm3.2-b} of Remark \ref{remarkdivisor}
with $c_1(i)= c_i[F_{i+1}:F_i]$ and $c_2(i)=[F_{i+1}:F_i]^{-1}$.

We recall now a standard result from the theory of constant field extensions (see \cite[Theorem 3.6.3]{Stichbook09}): let $\m{F}=\{F_i\}_{i=0}^{\infty}$ be a tower of function fields over $K$. By considering the constant field extensions $\bar{F_i}=F_i\cdot K'$  where $K'$ is an algebraic closure of $K$, we have the so called constant field extension tower $\bar{\m{F}}=\{\bar{F_i}\}_{i=0}^{\infty}$ of function fields over $K'$ and
\[\gamma(\m{F})=\gamma(\bar{\m{F}})\,.\]
Now we can prove the following result which will be useful later.
\begin{pro}\label{propmajwulf} Let $\m{F}=\{F_i\}_{i=0}^{\infty}$ be  a tower of function fields over $K$. Suppose that either each extension $F_{i+1}/F_i$ is Galois or that there exists a constant $M$ such that $[F_{i+1}:F_i]\leq M$ for $i\geq 0$. In order to have infinite genus it suffices to find, for infinitely many indices $i\geq 1$, a place $P_i$ of $F_0$ unramified in $F_i$ and such that each place of $F_i$ lying above $P_i$ is ramified in $F_{i+1}$.

In particular, suppose that the tower $\m{F}=\{F_i\}_{i=0}^{\infty}$ is a non skew recursive tower defined by a suitable polynomial $H\in K[x,y]$. Let   $\{x_i\}_{i=0}^{\infty}$ be a sequence of transcendental elements over $K$ such that $F_{i+1}=F_i(x_{i+1})$ where $H(x_{i+1},x_i)=0$. Then $\gamma(\m{F})=\infty$ if
\begin{enumerate}[(i)]
\item $H$, as a polynomial with coefficients in $K(y)$, is irreducible in $K(y)[x]$.
\item There exists an index $k\geq 0$  such that for infinitely many indices $i\geq 0$ there is a place $P_i$ of  $K(x_{i-k},\ldots,x_i)$ which is unramified in $F_i$ and each place of $F_i$ lying above $P_i$ is ramified in $F_{i+1}$.
\end{enumerate}
\end{pro}
\begin{proof}
We may assume that $K$ is algebraically closed since, by passing to the constant field tower $\bar{\m{F}}=\{\bar{F_i}\}_{i=0}^{\infty}$ with $\bar{F_i}=F_i\cdot K'$ where $K'$ is an algebraic closure of $K$, we have $\gamma(\m{F})=\gamma(\bar{\m{F}})$. In this situation we have that for each $i\geq 0$ the place $P_i$ of $F_0$ splits completely in $F_i$ and each place $Q$ of $F_i$ lying above $P_i$ ramifies in $F_{i+1}$. Now consider the following sets
\[R_i=\{P\in\P(F_i)\,:\,\text{$P$ is ramified in $F_{i+1}$}\}\,,\]
 and
\[A_i=\{Q\in\P(F_{i+1})\,:\,\text{$Q$ lies over some $P\in R_i$}\}\,.\]
and set $r_i=\#(R_i)$. Let $B_i$ be a divisor of $F_i$ defined as
\[B_i=\sum_{P\in R_i}P\,.\]
Then $\deg B_i \geq r_i\geq [F_i:F_0]$, because every place $Q$ of $F_i$ lying above $P_i$ is in $R_i$ and $P_i$ splits completely in $F_i$, so that condition \eqref{thm3.2-a} of Remark \ref{remarkdivisor} holds with $c_1(i)=1$.

Now suppose that each extension $F_{i+1}/F_i$ is Galois. Then $A_i$ is the set of all places of $F_{i+1}$ lying above a place of $R_i$. Therefore
\begin{align*}
   \sum_{P\in supp(B_i)}\underset{Q|P}{\sum_{Q\in\P(F_{i+1})}}d(Q|P)\deg Q & \geq \sum_{P\in R_i}\sum_{Q\in A_i}d(Q|P)\deg Q\\ & \geq  \frac{1}{2}\sum_{P\in R_i}\sum_{Q\in A_i}e(Q|P)f(Q|P)\deg P\\
    &=\frac{1}{2}[F_{i+1}:F_i]\, \sum_{P\in R_i}\deg P\\
    &\geq \frac{1}{2}[F_{i+1}:F_i]\,\deg B_i \,.
\end{align*}
Then  condition \eqref{thm3.2-b} of Remark \ref{remarkdivisor} holds with $c_2(i)=1/2$ and the series $\sum_{i=1}^{\infty}c_1(i)c_2(i)$ is divergent. Hence $\gamma(\m{F})=\infty$.
In the case that $[F_{i+1}:F_i]\leq M$ for $i\geq 0$ by taking $c_2(i)=M^{-1}$
we arrive to the same conclusion.

Finally suppose that the tower $\m{F}=\{F_i\}_{i=0}^{\infty}$ is  non skew and recursive. Since $\m{F}$ is non skew and $(i)$ holds, we have that $[F_i:F_0]=m^i=[F_i:K(x_i)]$ where $m=\deg_yH=\deg_xH$. Now we proceed with the same divisor $B_i$ as defined above using $(ii)$. We have that
\[\deg B_i\geq [F_i:K(x_{i-k},\ldots,x_i)]=m^{-k}[F_i:K(x_i)]=m^{-k}[F_i:F_0]\,,\]
so that by taking $c_1(i)=m^{-k}$ and $c_2(i)=m^{-k-1}$ we have the desired conclusion.
\end{proof}

An example of the  situation described in the second part of Proposition \ref{propmajwulf} for $k=0$ was given in Lemma 3.2 in \cite{MaWu05} and applied to the non skew Kummer tower
\[y^3=1-\left(\frac{x-1}{x+1}\right)^3\,,\]
over $\ff_p$ with $p\equiv 1,7\mod 12$. Unfortunately there is a mistake in the proof as we show now. The basic function field associated to that tower is $\ff_p(x,y)$ and both extensions $\ff_p(x,y)/\ff_p(x)$ and $\ff_p(x,y)/\ff_p(y)$ are Galois. The key part of the argument is that $-3^{-1}$ is not a square in $\ff_p$ with $p\equiv 1,7\mod 12$. With this we would have that the polynomial $x^2+3^{-1}$ is irreducible in $\ff_p[x]$ and then it would define the place $P_{x^2+3^{-1}}$ of $\ff_p(x)$ which is not only totally ramified in $\ff_p(x,y)$ (by the theory of Kummer extensions) but also of degree $2$, which is crucial for their argument. From these facts the authors deduce that the above equation defines a tower in the sense of Garcia and Stichtenoth with infinite genus. But any such prime is congruent to $1$ modulo $3$ and  $-3^{-1}$ is a square in $\ff_p$ for  $p\equiv 1\mod 3$ as can be easily seen using the quadratic reciprocity law. Thus the polynomial $x^2+3^{-1}$ is not irreducible in $\ff_p[x]$ so it does not define a place of $\ff_p(x)$.

\section{Climbing the pyramid}\label{pyramid}

In this section and the next one  we shall use the following convention: a place defined by a monic and irreducible polynomial $f\in K[x]$ in a rational function field $K(x)$ will be denoted by $P_{f(x)}$. A slight modification of the arguments given in Lemma 3.2 of \cite{MaWu05} allowed us to prove the following useful criterion for infinite genus in the case of recursive towers and we include the proof for the sake of completeness. The main difficulty on the applicability of Lemma 3.2 of \cite{MaWu05} is that it requires that both extensions $K(x,y)/K(x)$ and $K(x,y)/K(y)$ be Galois, which is something unusual or simply hard to prove. Getting rid of the condition $K(x,y)/K(y)$ being Galois was the key ingredient in proving the main result in the next section.
\begin{pro}\label{examplemawu} Let $\m{F}=\{F_i\}_{i=0}^{\infty}$ be a non skew recursive tower of function fields over $K$ defined by a polynomial $H\in K[x,y]$ with the same degree $m$ in both variables.  Let $K(x,y)$ be the basic function field associated to $\m{F}$ and consider the set
$$N=\{\deg{R} : R\in \P(K(y)) \text{ and $R$ is ramified in $K(x,y)$}\} \,. $$
Let $d\in \mathbb{N}$ such that $\gcd(d,m)=1$ and $n\ncon 0\mod d$ for all $n\in N$.
Suppose that there is a place $P$ of $K(x)$ with the following properties:
\newcounter{saveenum}
\begin{enumerate}[(a)]
\item $\deg{P}=d$ and\label{item-a-mawu}
\item $P$ is ramified in $K(x,y)$.\label{item-b-mawu}
\setcounter{saveenum}{\value{enumi}}
\end{enumerate}
Then $\gamma(F)=\infty$ if $K(x,y)/K(x)$ is a Galois extension and $H$, as a polynomial with coefficients in $K(y)$, is irreducible in $K(y)[x]$.
\end{pro}
\begin{proof}
Consider a sequence $\{x_i\}_{i=0}^{\infty}$ of transcendental elements over $K$ such that
\[F_0=K(x_0)\quad\text{and}\quad F_{i+1}=F_i(x_{i+1})\,,\]
where $H(x_i,x_{i+1})=0$ for $i\geq 0$. Let $i\geq 1$. By the above assumptions there is a place $P_i$ of $K(x_i)$ ramified in the extension $K(x_i, x_{i+1})/K(x_i)$ with $\deg{P_i}=d$.  Let $Q$ be a place of $F_i$ lying above $P_i$. Let  $P_0, P_1, \ldots, P_i$ be the restrictions of $Q$ to $K(x_0), K(x_1),\ldots, K(x_i)$ respectively and let $P'_j$ be a place of $K(x_j,x_{j+1})$ lying above $P_j$ for $j=1,\ldots i$ (see Figure \ref{figu5.9} below).

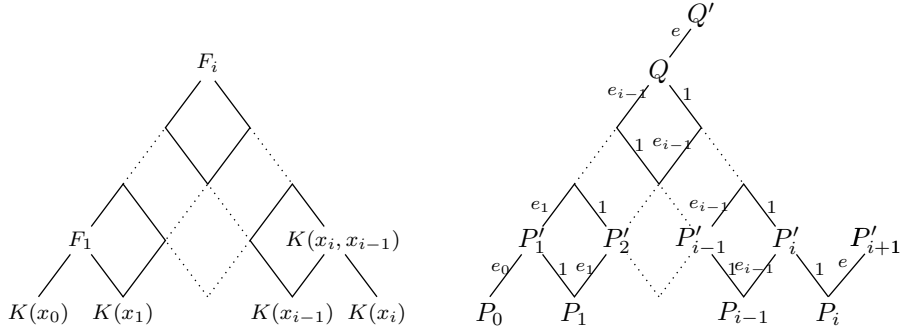
\begin{figure}[h!t]
 \begin{center}
  \begin{tikzpicture}[scale=0.75]
     \node at(0,-0.3){\footnotesize{$K(x_0)$}};
     \node at(1.5,-0.3){\footnotesize{$K(x_1)$}};
  \node at(4.5,-0.3){\footnotesize{$K(x_{i-1})$}};
  \node at(6,-0.3){\footnotesize{$K(x_i)$}};

   \node at(8,-0.3){$P_0$};
     \node at(9.5,-0.3){$P_1$};
  \node at(12.5,-0.3){$P_{i-1}$};
  \node at(14,-0.3){$P_i$};
  \draw[line width=0.5 pt](8,0)--(9.5,2)--(10.25,1)--(9.5,0)--(8.75,1);
  \draw[line width=0.5 pt,dotted](9.5,2)--(10.25,3);
    \draw[line width=0.5 pt,dotted](11,0)--(10.25,1)--(11,2)--(11.75,1)--(11,0);
    \draw[line width=0.5 pt](14,0)--(12.5,2)--(11.75,1)--(12.5,0)--(13.25,1);
\draw[line width=0.5 pt](14,0)--(14.85,1);
\draw[white, fill=white](14.85,1) circle (0.3 cm);
        \node at(14.85,1){$P_{i+1}'$};
      \draw[line width=0.5 pt,dotted](11.75,3)--(12.5,2);
        \draw[line width=0.5 pt](11.75,3)--(11,2)--(10.25,3)--(11.75,5);
\draw[line width=0.5 pt](11,4)--(11.75,3);
         \draw[white, fill=white](5.25,1) circle (0.3 cm);
        \draw[white, fill=white](8.75,1) circle (0.3 cm);
        \node at(8.75,1){$P_1'$};
          \draw[white, fill=white](10.25,1) circle (0.3 cm);
        \node at(10.25,1){$P_2'$};
          \draw[white, fill=white](11.75,1) circle (0.3 cm);
        \node at(11.75,1){$P_{i-1}'$};
         \draw[white, fill=white](13.25,1) circle (0.3 cm);
        \node at(13.25,1){$P_i'$};
        \draw[white, fill=white](11,4) circle (0.3cm);
\draw[white, fill=white](11.75,5) circle (0.3cm);
        \node at(11,4){$Q$};
\node at(11.75,5){$Q'$};
\draw[white, fill=white](14.85,1) circle (0.3 cm);
        \node at(14.85,1){$P_{i+1}'$};
        \node at(13.85,0.5){\scriptsize{$1$}};
         \node at(12.7,0.5){\scriptsize{$e_{i-1}$}};
          \node at(12.3,0.5){\scriptsize{$1$}};
          \node at(9.3,0.5){\scriptsize{$1$}};
         \node at(9.7,0.5){\scriptsize{$e_1$}};
         \node at(8.2,0.5){\scriptsize{$e_0$}};
          \node at(8.9,1.6){\scriptsize{$e_1$}};
           \node at(10,1.6){\scriptsize{$1$}};
           \node at(11.9,1.6){\scriptsize{$e_{i-1}$}};
           \node at(13,1.6){\scriptsize{$1$}};
           \node at(11.25,2.7){\scriptsize{$e_{i-1}$}};
           \node at(10.65,2.7){\scriptsize{$1$}};
            \node at(10.45,3.6){\scriptsize{$e_{i-1}$}};
           \node at(11.5,3.6){\scriptsize{$1$}};
\node at(11.3,4.7){\scriptsize{$e$}};
\node at(14.26,0.6){\scriptsize{$e$}};

 \draw[line width=0.5 pt](0,0)--(1.5,2)--(2.25,1)--(1.5,0)--(0.75,1);
  \draw[line width=0.5 pt,dotted](1.5,2)--(2.25,3);
    \draw[line width=0.5 pt,dotted](3,0)--(2.25,1)--(3,2)--(3.75,1)--(3,0);
    \draw[line width=0.5 pt](6,0)--(4.5,2)--(3.75,1)--(4.5,0)--(5.25,1);
      \draw[line width=0.5 pt,dotted](3.75,3)--(4.5,2);
        \draw[line width=0.5 pt](3.75,3)--(3,2)--(2.25,3)--(3,4)--(3.75,3);

        \draw[white, fill=white](0.75,1) circle (0.3 cm);
        \node at(0.75,1){\footnotesize{$F_1$}};
         \draw[white, fill=white](3,4) circle (0.2 cm);
        \node at(3,4.15){\footnotesize{$F_i$}};
         \draw[white, fill=white](5.25,1) circle (0.3 cm);
        \node at(5.4,1){\footnotesize{$K(x_i,x_{i-1})$}};

  \end{tikzpicture}
  \caption{Ramification of $P_0, P_1,\ldots P_i$ in the pyramid.}\label{figu5.9}
\end{center}\end{figure}

By hypothesis we have that $e(P'_i|P_i)=1$. On the other hand
\begin{equation}\label{inertiadegree}
f(P_j'|P_j)\deg{P_j}=\deg{P_j'}=f(P_j'|P_{j-1})\deg{P_{j-1}}\,,
\end{equation}
 for $1\leq j\leq i $ where $f(P_j'|P_j)$ and $f(P_j'|P_{j-1})$ are the respective inertia degrees. Since $d=\deg P_i$ and $\gcd(d,m)=1$ from \eqref{inertiadegree} for $j=i$ we must have that $d$ is a divisor of $\deg P_{i-1}$, otherwise there would be a prime factor of $d$ dividing $m$ because $K(x_{i-1},x_i)/K(x_{i-1})$ is Galois and in this case $f(P_i'|P_{i-1})$ is a divisor of $m$.
Continuing in this way using \eqref{inertiadegree} we see that $d$ is a divisor of $\deg P_j$ for $j=1,\ldots i$ and this implies, by hypothesis, that each place $P_j$ is unramified in the extension $K(x_{j-1},x_j)/K(x_j)$ for $j=1,\ldots i$.

We have now a ramification situation as in Figure \ref{figu5.9} below. By Abhyankar's Lemma (see \cite[Theorem 3.9.1]{Stichbook09}) it follows that $e(Q|P_i)=1$. Now let $Q'$ be a place of $F_{i+1}$ lying above $Q$ and let $P_{i+1}'$ be the restriction of $Q'$ to  $K(x_i,x_{i+1})$. Then $P_{i+1}'$ lies above $P_i$ and $e(P_{i+1}'|P_i)=e>1$ because $P_i$ is ramified in $K(x_i,x_{i+1})$ and the extension $K(x_i,x_{i+1})/K(x_i)$ is Galois. Once again, by Abhyankar's Lemma, we have that  $e(Q'|Q)=e(P'_{i+1}|P_i)>1$. Then we are in the conditions $(i)$ and $(ii)$ of Proposition \ref{propmajwulf} with $k=0$ and thus $\gamma(\m{F})=\infty$.

\end{proof}

\begin{rem}\label{remark-thm4-2}
Note that if we have  a ramification situation as in Figure \ref{figu5.9} above and $P_i$ is totally ramified in $K(x_i,x_{i+1})$ for all $i\geq 0$ then $Q$ is totally ramified in $F_{i+1}$ for all $i\geq 0$ because $e=[K(x_i,x_{i+1}):K(x_i)]=[F_{i+1}:F_i]$. Therefore if a recursive sequence $\m{F}$ of function fields is defined by a separable polynomial $H(x,y)$ in the second variable and for each $i\geq 0$ we have a ramification situation as in Figure \ref{figu5.9} and $P_i$ is totally ramified in $K(x_i,x_{i+1})$ for all $i\geq 0$ then $K$ is the full field of constants of each $F_i$ so that  $\m{F}$ is, in fact, a tower.
\end{rem}

\section{Classification of asymptotically good cubic towers of Garcia, Stichtenoth and Thomas type}\label{examples}
 We prove now our main result. As we said in the introduction Garcia, Stichtenoth and Thomas introduced in \cite{GST97} an interesting class of Kummer type towers  over a finite field $\ff_q$ with $q\equiv 1\mod m$ defined by an equation of the form
\begin{equation}
y^m=x^df(x)\,,
\end{equation}
where $f(x)$ is a polynomial of degree $m-d$ such that $f(0)\neq 0$ and $\gcd(d,m)=1$. These Kummer type towers have positive splitting rate  but over prime fields Lenstra \cite{Le02} showed that they fail to satisfy a well-known criterion for finite ramification locus given in \cite{GST97} which is the main tool in proving the finiteness of their genus. In this context the next result is important in the study of the cubic case of these Kummer type towers.

\begin{thm}\label{teoexe2}
  Let $p$ be a prime number and let $q=p^r$ with $r\in \na$ such that $q\equiv 1 \mod 3$. Let $f(t)=t^2+bt+c \in \ff_q[t]$ be a polynomial such that $t=0$ is not a double root. Let $\mathcal{F}$ be a Kummer type tower over $\ff_q$ recursively defined by the equation
  \begin{equation}\label{badkummerexample}
  y^3=xf(x)\,.
  \end{equation}
  If $\mathcal{F}$ is asymptotically good then the polynomial $f$ splits into linear factors over $\ff_q$. This implies that any asymptotically good tower recursively defined by \eqref{badkummerexample} is of one and only one of the following three types:
\begin{enumerate}[Type 1.]
\item  Recursively defined by $y^3=x(x+\alpha)(x+\beta)$ with non zero $\alpha\neq\beta\in\ff_q$. \label{a-teoexe2}
\item Recursively defined by $y^3=x^2(x+\alpha)$ with non zero $\alpha\in\ff_q$.  \label{b-teoexe2}
\item Recursively defined by $y^3=x(x+\alpha)^2$ with non zero $\alpha\in\ff_q$. \label{c-teoexe2}
\end{enumerate}
\end{thm}

\begin{proof}
 On the contrary, suppose that  the polynomial $f$ is irreducible over $\ff_q$. Let us consider the basic function field $F=\ff_q(x,y)$. Since the polynomial $f(x)$ is irreducible in $\ff_q[x]$ we have that the place $P_{f(x)}$ of $\ff_q(x)$ associated to $f(x)$ is of degree $2$ and, by the general theory of Kummer extensions (see \cite[Proposition 6.3.1]{Stichbook09}, $P_{f(x)}$ is totally ramified in $F$. In fact it is easy to see that the genus of $F$ is one and
\[\diff (F/\ff_q(x))=2Q_1+2Q_2\,,\]
where $Q_1$ is the only place of $F$ lying above $P_x$ (the zero of $x$ in $\ff_q(x)$) and $Q_2$ is the only place of $F$ lying above $P_{f(x)}$. Also $Q_1$ is of degree $1$ and $Q_2$ is of degree $2$.

The extension $F/\ff_q(y)$ is of degree $3$ because the polynomial
\[\phi(t)=tf(t)-y^3\in \ff_q(y)[t]\,,\]
is the minimal polynomial of $x$ over $\ff_q(y)$, otherwise $\phi(t)$ would have a root $z\neq y$ in $ \ff_q(y)$ and this would imply that $y$ is algebraic over $ \ff_q$, a contradiction.  Clearly the extension $F/\ff_q(y)$ is tame.

By choosing the place $P_{f(x)}$ of $\ff_q(x)$ we have that items \eqref{item-a-mawu} and \eqref{item-b-mawu} with $d=2$ hold  in Proposition~\ref{examplemawu} so it remains to prove that the integers in the set
$$N=\{\deg{R} : R\in \P(\ff_q(y)) \text{ and $R$ is ramified in $F$}\} \,, $$
are odd integers. We shall use the following notation: for $z\in F$ the symbols $(z)_F$, $(z)_0^F$ and $(z)_{\infty}^F$ denote the principal divisor, the zero divisor and the pole divisor of $z$ in $F$ respectively. Using the well known expression of the different divisor in terms of differentials (see Chapter $4$ of \cite{Stichbook09}) we have that
\begin{align}\label{diffequality}
\begin{split}
\diff (F/\ff_q(y)) &= 2(y)_{\infty}^F+(dy)_F\\
&= 2(y)_{\infty}^F+\left(\frac{f(x) + x f'(x)}{3y^2}\right)_F+(dx)_F\\
&=2(y)_{\infty}^F+\left(\frac{(x-\beta_1)(x-\beta_2)}{y^2}\right)_F-2(x)_{\infty}^F+\diff (F/\ff_q(x))\\
& =2(y)_{\infty}^F+\left(\frac{(x-\beta_1)(x-\beta_2)}{y^2}\right)_F-2(x)_{\infty}^F+2Q_1+2Q_2\,.
\end{split}
\end{align}
We show now that $(y)_{\infty}^F=(x)_{\infty}^F$. Let $Q\in \supp (y)_{\infty}^F$ and let $S=Q\cap \ff_q(x)$. Then
\[3\nu_Q(y)=e(Q|S)((\nu_S(x)+\nu_S(f(x)))\,.\]
Since $\nu_Q(y)<0$ we must have that $S=P_{\infty}^x$, the pole of $x$ in  $\ff_q(x)$. Hence $\nu_Q(y)=-e(Q|P_{\infty}^x)=-1$ because by Kummer theory (see \cite[Proposition 6.3.1]{Stichbook09}) $P_{\infty}^x$ is unramified in $F$. Then
\[-3=\nu_Q(y^3)=\nu_Q(x)+\nu_Q(f(x))\,,\]
and this implies that $\nu_Q(x)<0$. Therefore $-3=3\nu_Q(x)$ and we have $\nu_Q(x)=-1$ which says that $Q\in \supp (x)_{\infty}^F$ and $\nu_Q(\supp (y)_{\infty}^F)=\nu_Q(\supp (x)_{\infty}^F)$.

Reciprocally let $Q\in \supp (x)_{\infty}^F$. Since $\nu_Q(x)<0$ we have
\[3\nu_Q(y)=\nu_Q(x)+\nu_Q(f(x))=3\nu_Q(x)\,,\]
so that $\nu_Q(y)=\nu_Q(x)<0$. If $S=Q\cap \ff_q(x)$ then
\[3\nu_Q(y)=e(Q|S)((\nu_S(x)+\nu_S(f(x)))\,,\]
and we must have again that $S=P_{\infty}^x$. This implies that $\nu_Q(y)=-e(Q|P_{\infty}^x)=-1$. Therefore $Q\in \supp (y)_{\infty}^F$ and $\nu_Q(\supp (x)_{\infty}^F)=\nu_Q(\supp (y)_{\infty}^F)$. Hence $(y)_{\infty}^F=(x)_{\infty}^F$ as claimed.

From \eqref{diffequality} we have now that
\begin{align}\label{diffequalityreduced}
\begin{split}
\diff (F/\ff_q(y))&=\left(\frac{(x-\beta_1)(x-\beta_2)}{y^2}\right)_F+2Q_1+2Q_2\\&=(z)_0^F-(z)_\infty^F+2Q_1+2Q_2\,,
\end{split}
\end{align}
where $z=(x-\beta_1)(x-\beta_2)y^{-2}$.

Let $Q$ be a place of $F$ in the support of $(z)_0^F$. Then $\nu_Q(z)>0$ and thus one of the following two cases can occur:
\begin{enumerate}[(i)]
  \item $\nu_Q(x-\beta_i)>0$ for $i=1$ or $i=2$. In either case $Q$ lies above the rational place $P_{x-\beta_i}$ of $\ff_q(x)$. Since $F/K(x)$ is a Galois extension of degree $3$ and $\deg Q=f(Q|P)\deg P_{x-\beta_i}$ we have that either $\deg Q =1$ or $\deg Q =3$.
  \item $\nu_Q(y)<0$. Let $S=Q\cap \ff_q(x)$. We have $$3\nu_Q(y)=e(S|Q)(\nu_S(x)+\nu_S(f(x)))\,.$$ Since $\nu_S(x)\geq 0$ leads to a contradiction we must have $\nu_S(x)<0$ and thus $S=P_\infty^x$.  The same argument used in (i) above shows that either $\deg Q =1$ or $\deg Q =3$.
\end{enumerate}

Now let $Q$ be a place of $F$ in the support of $(z)_\infty^F$. Then $\nu_Q(z)<0$ and thus one of the following  two cases can occur:
\begin{enumerate}[(a)]
  \item $\nu_Q(x-\beta_i)<0$ for $i=1$ or $i=2$. In either case $\nu_Q(x)<0$ so that $Q$ lies above the place $P_\infty^x$ of $\ff_q(x)$ and the same argument given in (i) above shows that either $\deg Q =1$ or $\deg Q =3$.
  \item $\nu_Q(y)>0$. Let $S=Q\cap \ff_q(x)$. We have
  \begin{equation}\label{case2}
  3\nu_Q(y)=e(S|Q)(\nu_S(x)+\nu_S(f(x)))\,.
  \end{equation}
  Since $\nu_S(x)< 0$ leads to a contradiction we must have that $\nu_S(x)\geq 0$. If $\nu_S(x)> 0$  then $S=P_x$ and so $Q=Q_1$. If $\nu_S(x)=0$ then we must have that $\nu_S(f(x))>0$ because the left hand side of \eqref{case2} is positive. Therefore if $\nu_S(x)=0$ then $S=P_{f(x)}$ and thus $Q=Q_2$.
\end{enumerate}

On the other hand $\nu_{Q_i}(y)=1$ for $i=1,2$ as it is easy to see from the definition of each $Q_i$. Then $\nu_{Q_i}(z)=-2\nu_{Q_i}(y)=-2$ so that, in fact, the divisor $-2Q_1-2Q_2$ is part of the divisor $(z)_F$. This implies that both places $Q_1$ and $Q_2$ are not in the support of $\diff (F/\ff_q(y))$. From the cases (i), (ii) and (a) above we conclude that every place in the support of $\diff (F/\ff_q(y))$ is of degree $1$ or $3$. Therefore no place of even degree en $\ff_q(y)$ can ramify in $F$ as we claimed. In this way we see that all the conditions of Proposition \ref{examplemawu} hold so that the equation
\[y^3=xf(x)\,,\]
defines a Kummer tower $\m{F}$ over $\ff_q$ with infinite genus if $f(x)$ is irreducible over $\ff_q$ and this proves the theorem.
\end{proof}

\bibliographystyle{plain}{}

\begin{thebibliography}{1}

\bibitem{GS07}
A.~Garcia and H.~Stichtenoth.
\newblock Explicit towers of function fields over finite fields.
\newblock In {\em Topics in geometry, coding theory and cryptography}, volume~6
  of {\em Algebr. Appl.}, pages 1--58. Springer, Dordrecht, 2007.

\bibitem{GST97}
A.~Garcia, H.~Stichtenoth, and M.~Thomas.
\newblock On towers and composita of towers of function fields over finite
  fields.
\newblock {\em Finite Fields Appl.}, 3(3):257--274, 1997.

\bibitem{Ha05}
T.~Hasegawa.
\newblock An upper bound for the {G}arcia-{S}tichtenoth numbers of towers.
\newblock {\em Tokyo J. Math.}, 28(2):471--481, 2005.

\bibitem{Le02}
H.~W. Lenstra, Jr.
\newblock On a problem of {G}arcia, {S}tichtenoth, and {T}homas.
\newblock {\em Finite Fields Appl.}, 8(2):166--170, 2002.

\bibitem{MaWu05}
H.~Maharaj and J.~Wulftange.
\newblock On the construction of tame towers over finite fields.
\newblock {\em J. Pure Appl. Algebra}, 199(1-3):197--218, 2005.

\bibitem{NX01}
H.~Niederreiter and C.~Xing.
\newblock {\em Rational points on curves over finite fields: theory and
  applications}, volume 285 of {\em London Mathematical Society Lecture Note
  Series}.
\newblock Cambridge University Press, Cambridge, 2001.

\bibitem{Stichbook09}
H.~Stichtenoth.
\newblock {\em Algebraic function fields and codes}, volume 254 of {\em
  Graduate Texts in Mathematics}.
\newblock Springer-Verlag, Berlin, second edition, 2009.

\end{thebibliography}

%
%

\end{document}